\documentclass[10pt]{amsart}

\usepackage[left=2.5cm,top=2.5cm,right=2.5cm,bottom=2.5cm]{geometry}

\theoremstyle{plain}
\newtheorem{thm}{Theorem}
\newtheorem*{thm*}{Theorem}
\newtheorem{lemma}[thm]{Lemma}

\newtheorem*{lem*}{Lemma}

\newtheorem*{prop*}{Proposition}

\newtheorem*{cor*}{Corollary}

\newtheorem{conj}[thm]{Conjecture}
\newtheorem*{conj*}{Conjecture}

\theoremstyle{definition}

\newtheorem*{cons*}{Construction}

\newtheorem*{df*}{Definition}

\newtheorem*{nota*}{Notation}
\newtheorem*{problem*}{Problem}

\newtheorem*{qu*}{Question}

\newtheorem{rmk}[thm]{Remark}
\newtheorem*{rmk*}{Remark}

\newtheorem*{ex*}{Example}

\usepackage{url}
\usepackage{todonotes}
\usepackage{longtable}
\usepackage{hyperref}
\hypersetup{
  colorlinks   = true, 
  urlcolor     = blue, 
  linkcolor    = blue, 
  citecolor   = purple 
}
\usepackage{pdflscape}

\DeclareMathOperator{\Li}{Li}

\newcommand{\floor}[1]{\left\lfloor #1 \right\rfloor}

\begin{document}

\title{The Least Prime in Arithmetic an Progression}
\author{Andrew Fiori}
\address{University of Lethbridge}
\email{andrew.fiori@uleth.ca}
\thanks{Andrew Fiori  acknowledges the support of NSERC Discovery Grant RGPIN-2020-05316.}
\thanks{The computations for this work were conducted on systems supported by the University of Lethbridge and purchased through an NSERC RTI grant.}

\begin{abstract}
In his 1979 paper Samuel Wagstaff studied the problem of bounding the first prime in an arithmetic progression. 
In this paper we update a number of his computations using advances in hardware.
Based on this we refine his conjecture on Primes in Arithmetic Progression and provide further numerical evidence in support of it. 
For instance we conjecture that for $n>3$ we have a bound of $3\phi(n)\log(n)\log(\phi(n))$ and verified this bound to $10^8$.
\end{abstract}
\maketitle

\section{Introduction}

In his 1979 paper, \cite{WAG79}, Samuel Wagstaff studied the problem of bounding the first prime in an arithmetic progression. 
His paper provides a heuristic and some limited data to justify a new conjectural upper bound on when the first prime would occur. 
In the present work we aim to refine the conjecture and provide further data in support of it. 
The computations would also be expected to have applications towards an explicit version of Linnik's theorem holding for all moduli as such a result almost certainly requires checking small moduli directly.

Given a modulus $n$ and a number $a<n$ with ${\rm gcd}(a,n)=1$ we shall denote by 
\[ P(n,a) := {\rm min}( \{ p\; \text{prime}\;|\;  p = a \pmod{n} \})  \]
and by
\[ P(n) := \underset{(a,n)=1}{\rm max} P(n,a). \]
By Linnik's theorem \cite{Linnik} it is known that there exist an absolute constant $A$ such that $P(n) < n^A$. 
For $n$ sufficiently large it has been shown by Xylouris \cite{Xylouris} that one may take $A=5.2$, and it has  been conjectured by Kanold \cite{Kanold} that for all $n$ one has $A=2$. No fully explicit result holding for all $n$ is known.

Assuming the Generalized Riemann Hypothesis for $L$-functions of characters of modulus $n$ Titchmarsh \cite{Titchmarsh} showed that $P(n) <\!\!< \phi(n)^2\log(n)^4$. 
Meanwhile Turan \cite{Turan} showed, under the same hypothesis, that for all $\delta > 0$ we have that $P(n,a) <\!\!< n\log(n)^{2+\delta}$ except for at most $o(\phi(n))$ many counterexamples.
Heath-Brown \cite{HB} speculated that $P(n) <\!\!<  n\log(n)^2$ whereas Wagstaff more precisely conjectured that $P(n) <\!\!<  \phi(n)\log(n)\log(\phi(n))$.

We note that lower bounds on $P(n)$ have also been studied, see for example \cite{Erdos, Landau, Prachar, Schinzel, Pomerance}. 

\section{Heuristics for Bounds}

We shall briefly outline Wagstaff's heuristic.
Fix a function $m=m(n) <\!\!< \log(n)^3$ and set $X = X(n) =m(n) n\log(n)$.
For a given $a$ and $n$ Wagstaff estimates the probability that a number, $a+nk$ for $k=0,\ldots, \floor{X/n}$, is prime is approximately:
\begin{equation}\label{wagprime}  \frac{n}{\phi(n)\log(X)}. \end{equation}
Based on this (and assuming independence) one finds the probability that $P(n,a) > X$ is approximately:
\begin{equation}\label{wagpna}  \left( 1- \frac{n}{\phi(n)\log(X)}\right)^{m\log(n)}  \approx e^{-mn/\phi(n)}. \end{equation}
Hence for a given $n$ then one expects that the probability that $P(n) < X$ is approximately
\begin{equation}\label{wagp} (1- e^{-mn/\phi(n)})^{\phi(n)}.\end{equation}
Based on this Wagstaff makes the following conjecture:
\begin{conj}[Wagstaff]
As $n\rightarrow \infty$ for most $n$ we have 
\[ P(n) \sim \phi(n)\log(n)\log(\phi(n)). \]
\end{conj}
Based on his heuristic we will give a stronger conjecture. 

\begin{rmk}\label{rmk:error}
The formula of equation \eqref{wagprime} can be thought of as the average probability that one of those numbers is prime. That is, by the prime number theorem the number of primes in that list should be approximately $\frac{1}{\phi(n)}\frac{X}{\log(X)}$ and there are $X/n$ numbers in the list.
We briefly revisit the formula in order to refine the approximation.
For each individual number, $a+nk$ for $k=0,\ldots, \floor{X/n}$, the probability it is prime is approximately
\begin{equation}\label{aprime}   \frac{n}{\phi(n)}\frac{1}{\log( a+kn )}. \end{equation}
It follows that the probability that none of them is prime is approximately:
\[  Y = \prod_{k=0}^{\floor{X/n}} \left( 1 -  \frac{n}{\phi(n)}\frac{1}{\log( a+kn )} \right). \]
We estimate $Y$ by
\begin{align*}
 \log(Y)  
 &= \sum_{k=0}^{\floor{X/n}} \log\left( 1 -  \frac{n}{\phi(n)}\frac{1}{\log( a+kn )} \right) \\
  &= \sum_{k=0}^{\floor{X/n}} \sum_{i=1}^\infty - \frac{1}{i}\left( \frac{n}{\phi(n)}\frac{1}{\log( a+kn )} \right)^i \\
  &\sim \sum_{k=0}^{\floor{X/n}} \sum_{i=1}^3 - \frac{1}{i}\left( \frac{n}{\phi(n)}\frac{1}{\log( a+kn )} \right)^i \\
  &\sim \frac{-1}{n}\frac{n}{\phi(n)}\int_{2}^X  \left(\frac{1}{\log(x)} +\frac{n}{2\phi(n)}\frac{1}{\log(x)^2} + \frac{n^2}{3\phi(n)^2}\frac{1}{\log(x)^3} \right){\rm dx}\\
  &\sim\frac{-1}{\phi(n)} \left( \Li(x) + \frac{n}{2\phi(n)}\left(\Li(X) - \frac{X}{\log(X)}\right) + \frac{n^2}{6\phi(n)^2}\left(\Li(X) -\frac{X}{\log(X)}-\frac{X}{\log(X)^2}\right) \right) \\
    &\sim\frac{-1}{\phi(n)} \left( \frac{X}{\log(X)} + \frac{X}{\log(X)^2} + \frac{2X}{\log(X)^3}  + \frac{n}{2\phi(n)}\left( \frac{X}{\log(X)^2} + \frac{2X}{\log(X)^3} \right) + \frac{n^2}{6\phi(n)^2}\left( \frac{2X}{\log(X)^3} \right) \right) \\
    &\sim\frac{-1}{\phi(n)} \left( \frac{X}{\log(X)} +\left(1+\frac{n}{2\phi(n)}\right)\frac{X}{\log(X)^2} + \left(1+\frac{n}{2\phi(n)}+\frac{n^2}{6\phi(n)^2}\right)\frac{2X}{\log(X)^3}  \right)  \\
 &\sim \frac{-X}{\log(X)\phi(n)}\left( 1 +\left(1+\frac{n}{2\phi(n)}\right)\frac{1}{\log(X)} + \left(1+\frac{n}{2\phi(n)}+\frac{n^2}{6\phi(n)^2}\right)\frac{2}{\log(X)^2}\right) \\
  &\sim \frac{-X}{\log(X)\phi(n)}
 \end{align*}
We thus might be inclined to estimate this by $\frac{-X}{\log(X)\phi(n)}$ and thus estimate $Y$ as:
\[ Y \approx e^{-X/(\log(X)\phi(n))}. \]
If we write $X=m(n) n \log(n)$ with $m(n) <\!\!< \log(n)^3$ then we have $\log(X) \sim \log(n)$ and 
\[ X/(\log(X)\phi(n)) \sim m(n)n/\phi(n)  \]
and we recover Wagstaff's formula. 

Now, we will be considering the case $X=X(n)=(1+\epsilon)\phi(n)\log(n)\log(\phi(n))$, so the case where $m(n) = (1+\epsilon)\log(\phi(n))\phi(n)/n$. If we put this directly into Wagstaff's formula we obtain:
\[   e^{-mn/\phi(n) } =\frac{1}{\phi(n)^{1+\epsilon}}. \]
However, in this case it is more accurate to write $\log(X) \sim \log(\phi(n))+2\log\log(n)$\footnote{We note that $\log(X) >\!\!> \log(n)$}  and we obtain
\begin{align*}
 X/(\log(X)\phi(n)) &= (1+\epsilon)\phi(n)\log(n)\log(\phi(n))/\phi(n)(\log(\phi(n))+2\log\log(n)) \\
 &\sim  (1+\epsilon)\log(n)/(1+\tfrac{2\log\log(n)}{\log(\phi(n))}) \\
 &\sim  (1+\epsilon)\log(n)  
 \end{align*}
so that we may instead approximate
\begin{equation}\label{apna}  e^{-X/(\log(X)\phi(n)) }  \approx \frac{1}{n^{1+\epsilon}}. \end{equation}

It is important to also note some subtleties in taking exponentials of asymptotic formulas.
For example the denominator we dropped of $(1+\tfrac{2\log\log(n)}{\log(\phi(n))})$ is relatively far from $1$ for small $n$.
One can estimate better
\begin{align*} \log(n)/(1+\tfrac{2\log\log(n)}{\log(\phi(n))})  
&\sim \log(n) - \tfrac{2\log\log(n)\log(n)}{\log(\phi(n))} + \tfrac{4\log\log(n)^2\log(n)}{\log(\phi(n))^2} + \cdots \\
&\sim  \log(n) - 2\log\log(n) +  \tfrac{4\log\log(n)^2}{\log(n)} + \cdots 
\end{align*}
Upon exponentiation this might yield an additional factor of $\log(n)^{2+2\epsilon}$. 

Additionally, while estimating $\log(Y)$ the error term we have discarded of size
\begin{align*}
  \frac{-X}{\log(X)\phi(n)} \left(1 + \frac{n}{2\phi(n)}\right)\frac{1}{\log(X)} 
  &=
 \frac{-(1+\epsilon)\phi(n)\log(\phi(n))\log(n)}{(\log(\phi(n))+2\log\log(n))\phi(n)} \left(1 + \frac{n}{2\phi(n)}\right)\frac{1}{(\log(\phi(n))+2\log\log(n))} \\
 &=
  \frac{-(1+\epsilon)}{1+\frac{2\log\log(n)}{\log(\phi(n))}} \left(1 + \frac{n}{2\phi(n)}\right)\frac{\log(n)}{(\log(\phi(n))+2\log\log(n))} \\
 &\sim
-(1+\epsilon)\left(1 + \frac{n}{2\phi(n)}\right)
\end{align*}
is potentially non-negligible once we exponentiate it.
In particular, if $\frac{n}{\phi(n)}$ is large\footnote{Recall that $1<n/\phi(n) <\!\!< \log\log(n)$}, then, it becomes less likely that $P(n)$ is larger than $X$ by a factor of up to $\frac{1}{e\log(n)^{1/2}}$. 
We note that the lower order terms, for instance of size $O(\frac{1}{\log(X)^3})$, in the exponent will tend not to influence the asymptotic.
However any other errors of constant size, for instance coming from the choice of left end point of integration, will multiply the final formula by a constant.

The final takeaway from these calculations is that the probability $P(n,a) > (1+\epsilon)\phi(n)\log(n)\log(\phi(n))$ should be between 
\[ \frac{c_1}{\log(n)^{1/2}n^{1+\epsilon}} \quad\text{ and }\quad \frac{c_2\log(n)^{2+2\epsilon}}{\phi(n)^{1+\epsilon}}. \]
A good estimate might be of the form:
\[ c_3\left( \frac{1}{e^{1+\frac{n}{2\phi(n)}}n}\right)^{\frac{1+\epsilon}{1+\frac{2\log\log(n)}{\log(n)}}}. \]
\end{rmk}

\section{Counting Outliers}

For $\epsilon>0$ and $x>0$ we define sets of $n$ for which $P(n)$ is a large outlier, that is:
\[ E_\epsilon = \{  n  \;|\;  P(n) > (1+\epsilon) \phi(n)\log(n)\log(\phi(n))\} \]
and 
\[ E_{\epsilon}(x,y) =  \{ x <  n < y \;|\;  P(n) > (1+\epsilon) \phi(n)\log(n)\log(\phi(n))\}. \]
Based on \eqref{wagpna}  one finds the probability that $n \in  E_{\epsilon}$ is approximately:
\begin{equation}\label{wagpf} 1-\left(1- \frac{1}{\phi(n)^{1+\epsilon}}\right)^{\phi(n)} \end{equation}
whereas based on \eqref{apna} it would be
\begin{equation} \label{ap} 1-\left(1- \frac{1}{n^{1+\epsilon}}\right)^{\phi(n)}. \end{equation}

We shall make use of the following lemma:
\begin{lemma}
For any $\epsilon > 0$ and for any $n\ge 2$ we have:
\[    \frac{1}{\phi(n)^{\epsilon}} -\frac{1}{2\phi(n)^{2\epsilon}} \leq  1-\left(1-\frac{1}{\phi(n)^{(1+\epsilon)}}\right)^{\phi(n)}  \leq \frac{1}{\phi(n)^{\epsilon}} \]
and
\[    \frac{\phi(n)}{n}\frac{1}{n^{\epsilon}}  - \frac{1}{2n^{2\epsilon}} \leq  1-\left(1-\frac{1}{n^{(1+\epsilon)}}\right)^{\phi(n)}  \leq \frac{\phi(n)}{n}\frac{1}{n^{\epsilon}} \leq  \frac{1}{n^{\epsilon}}. \]

\end{lemma}
\begin{proof}
Apply the binomial theorem to 
\[ \left(1-\frac{1}{\phi(n)^{(1+\epsilon)}}\right)^{\phi(n)}  = 1 - \phi(n) \frac{1}{\phi(n)^{(1+\epsilon)}} + \sum_{k=2}^{\phi(n)} (-1)^k \begin{pmatrix} \phi(n) \\ k \end{pmatrix}  \frac{1}{\phi(n)^{k(1+\epsilon)}}. \]
We notice that the absolute value of the terms $a_k$ in the summation are decreasing and alternating hence any pair of terms $a_{2\ell}+a_{2\ell+1} > 0$ and hence (using also that the last term is positive) we have:
\[  \left(1-\frac{1}{\phi(n)^{(1+\epsilon)}}\right)^{\phi(n)} >  1 - \phi(n) \frac{1}{\phi(n)^{(1+\epsilon)}}. \]
It follows that 
\[  1-\left(1-\frac{1}{\phi(n)^{1+\epsilon}}\right)^{\phi(n)}  < \frac{1}{\phi(n)^{\epsilon}}. \]
The argument for the lower bound is similar. Except one additionally groups the third and final terms of the sum.
\end{proof}
\begin{rmk}
Similar results hold if one replaces $\frac{1}{\phi(n)^{1+\epsilon}}$ or $\frac{1}{n^{1+\epsilon}}$ with other similar formulas for instance $\frac{1}{\log(n)^c n^{1+\epsilon}}$.
\end{rmk}

Based on the above it would seem reasonable to conjecture, for example, that for $\epsilon >0$ as $y \to\infty$ we have 
\[ |E_\epsilon(x,y)|  \sim  \left(\sum_{n=x}^{y}  \left( 1-\left(1-\frac{1}{n^{1+\epsilon}}\right)^{\phi(n)}\right)\right)   \sim      \sum_{n=x}^{y}    \frac{\phi(n)}{n}\frac{1}{n^{\epsilon}}.  \]
However, as noted in Remark \ref{rmk:error} we expect that our estimates on probabilities could be off by a factor as large as $\log(n)^2$. In Table \ref{table:stats-large} we will illustrate how well several different formulas work to estimate these counts.
That said, based on bounds on these probabilities we are able to make the following conjecture:
\begin{conj}\label{conjbig}
Let $\epsilon>0$. Then:
\begin{enumerate}
\item For $\epsilon > 1$ as $y \to\infty$ we have
 \[      \frac{1}{(1-\epsilon)\log(y)}(y^{1-\epsilon} - x^{1-\epsilon}) <\!\!<  |E_\epsilon(x,y)| <\!\!<  \frac{1}{\epsilon-1}(\log(x)^2x^{1-\epsilon}).  \]
 \item For $\epsilon=1$ as $y \to\infty$ we have
 \[    \log(\log(y)) <\!\!<  |E_\epsilon(x,y)| <\!\!< (\log(y)^3 - \log(x)^3). \]
  \item For $\epsilon > 1$ as $y \to\infty$ we have
 \[      \frac{1}{(1-\epsilon)\log(y)}(y^{1-\epsilon} - x^{1-\epsilon}) <\!\!<  |E_\epsilon(x,y)| <\!\!<  \frac{1}{1-\epsilon}(\log(y)^2y^{1-\epsilon} - \log(x)^2x^{1-\epsilon}).  \]
 
\item $\underset{n\to\infty}{\limsup} \frac{P(n)}{\phi(n)\log(n)\log(\phi(n))} = 2$.
\end{enumerate}
In particular, for all $\epsilon > 1$ we have that $ E_\epsilon$ is finite and for all $\epsilon \leq 1$ we have that $E_\epsilon$ is infinite
\end{conj}

\begin{rmk}\label{rmk:independence}
Passing from the estimates for a single $n$ to a conjecture about number of outliers involves an assumption of independence for different values of $n$ and $m$.
Though such an assumption is very reasonable for $n$ and $m$ relatively prime it certainly does not hold for $m=2n$ where $n$ is odd because in this case $\phi(n)=\phi(m)$ and virtually all of
the congruence classes under consideration are effectively the same.
However, provided $2{\rm gcd}(m,n) < m,n$ then it is highly unlikely that any arithmetic progressions modulo ${\rm gcd}(m,n)$ will control the size of either $P(m)$ or $P(n)$.

Indeed, even if we consider the case where half of congruences modulo $n$ have small primes coming from those modulo $d$, for some $d|n$ then considering the remaining classes gives
\[ 1-\left(1-\frac{1}{n^{1+\epsilon}}\right)^{\phi(n)/2} \sim  \frac{\phi(n)}{2}\frac{1}{n^{1+\epsilon}}. \]
And the expectation of being less than $(1+\epsilon)$ is only reduced by half. One can bound the expected impact of this as
\[   \prod_{p|n}\left( 1 - \frac{1}{p} \right) \]
which is bounded below by roughly $\frac{1}{\log\log(n)}$ but is only rarely this small.
One can see in Table  \ref{table:smalloutliers} the impact having many small prime factors has.
\end{rmk}

We can similarly study small outliers. That is for $0<\epsilon < 1$ we could introduce: 
\[ F_\epsilon = \{  n  \;|\;  P(n) < (1-\epsilon) \phi(n)\log(n)\log(\phi(n))\} \]
and 
\[ F_{\epsilon}(x,y) =  \{ x <  n < y \;|\;  P(n) < (1-\epsilon) \phi(n)\log(n)\log(\phi(n))\}. \]
Here we may estimate the probability that $n\in F_\epsilon$ as
\begin{equation}\label{smallest}
\left(1- \frac{1}{n^{1-\epsilon}}\right)^{\phi(n)} =  \left(\left(1- \frac{1}{n^{1-\epsilon}}\right)^{n^{1-\epsilon}} \right)^{\phi(n)/n^{1-\epsilon}} \sim e^{- \phi(n)/n^{1-\epsilon}} < e^{-n^{\epsilon}/\log\log(n)}. 
 \end{equation}
 
 \begin{rmk}
 Recalling Remark \ref{rmk:error} we note that better upper or lower bounds could easily be any number of things between
 \[   e^{-cn^{\epsilon}/\log(n)\log\log(n)}\quad\text{ and }\quad e^{-c\log\log(n) n^{\epsilon}}. \]
 We illustrate the numerics of several options in Table \ref{table:stats-small}.

 We do observe that in Table \ref{table:smalloutliers} small outliers often have a large number of small prime factors which agrees with the expectation from \eqref{smallest} that having $\phi(n)/n$ small would increase the probability of this.
 
 It is also  important to note that the approximations in \eqref{smallest} are not all that good for small $n$ so that it is not so surprising Table \ref{table:stats-small} suggests these probabilities are significant over-estimates.
 \end{rmk}  
 
 Based on this we conjecture:
\begin{conj}\label{conjsmall}
Let $1>\epsilon > 0$. Then
 \begin{enumerate}
 \item As $y\to\infty$ we have
 \[  F_{\epsilon}(x,y)  <\!\!< \int_{x}^y e^{-n^{\epsilon}/\log(n)\log\log(n)} {\rm dn}  \]
 \item
  $\underset{n\to\infty}{\liminf} \frac{P(n)}{\phi(n)\log(n)\log(\phi(n))} = 1$.
    \end{enumerate}
  In particular, for all $1>\epsilon>0$ we have that $ F_\epsilon$ is finite.
\end{conj}

 Based on Conjectures \ref{conjbig} and \ref{conjsmall} as well as Tables \ref{table:stats-large}, \ref{table:stats-small}, \ref{table:largeoutliers}, and \ref{table:smalloutliers} we can make the following further conjectures:
 \begin{conj}
 For all $n>3$ we have
 \[  P(n) < 3 \phi(n)\log(n)\log(\phi(n)) \]
 and for all $n>570$ we have
 \[  P(n) > 0.5 \phi(n)\log(n)\log(\phi(n)). \]
 \end{conj}
 \begin{rmk}
 One could certainly be bolder and predict that
 \begin{enumerate}
 \item for all $n>6$ that $P(n) < 2.4 \phi(n)\log(n)\log(\phi(n))$,
 \item for all $n>6840$ that  $P(n) > 0.6 \phi(n)\log(n)\log(\phi(n)) $, or
 \item for all $n>198660$  that  $P(n) > 0.7 \phi(n)\log(n)\log(\phi(n)) $.
 \end{enumerate}
 However, based on our expectations it is also reasonable to think that there may well be at least one further counterexample to each of these. For example predicting for all $n>1623$ that $P(n) < 2.2 \phi(n)\log(n)\log(\phi(n))$, seemed reasonable until we had checked $n$ past $8\cdot 10^6$.
 \end{rmk}

\section{Computing $P(n)$}

There are two natural approaches to computing $P(n)$.
\begin{enumerate}
\item For each $a<n$ with ${\rm gcd}(a,n)=1$ compute $P(n,a)$ by searching for a prime $p=a+kn$. 
         Computing $P(n,a)$ should involve checking $\log n$ many numbers and, conjecturally, a worst case of $(\log n)^2$.
         
         Verifying if a number is prime has running time of (at least) $(\log n)$.
         Consequently the algorithm has an expected running time of $O(\phi(n)(\log n)^2)$.
         
         One can perform a time/space trade off to make the primality test $O(1)$ at the expense of precomputing a list of primes. This requires $O(n(\log(n))^2)$ space. The expected running time becomes $O(\phi(n)(\log n))$.
         
\item Iterate over all primes and check if $p\pmod{n}$ is a new residue class until all residue classes are seen (ie: $\phi(n)$ of them).

         One expects to need to iterate over primes up to at most $2\phi(n)(\log(n))^2$, There are roughly $n\log(n)$ such primes, and it is beneficial to precompute this list if you are computing many values of $P(n)$.
         
         Assuming precomputation this approach has space complexity $n(\log(n))$ and running time $n\log(n)$.
         
\end{enumerate}     

One can also hybridize the two approaches:

         Iterate over primes until you have encountered all except $\phi(n)/\log(n)$ of the congruence classes.  This step has expected running time $O(\phi(n)\log(\phi(n)))$.
         
         For the remaining congruence classes compute $P(n,a)$ as above. This step has an expected running time $O(\phi(n)\log(n))$.

Regardless of the approach taken, the expected\footnote{this expectation is based on our conjectural answer to the problem being studied, this conjecture held in practice.} total running time to compute $P(n)$ up to $N$ is then $O(N^2(\log N))$ using $O(N(\log N)^2)$ space.

\begin{rmk} 
\begin{enumerate}
\item
All of these can be parallelized for different $n$, sharing any of the the precomputed prime lists. 
\item
The hybrid option was the fastest experimentally. The intuition is that it attempts to minimize the wait time to finding a new prime in a congruence class.
\item 
The time space trade offs become intractable for large $n$ but note that because one can use a bitset to record which $p<3N(\log(N))^2$ are prime actually storing the list of the first $O(N\log(N))$ primes will, for smaller values of $N$, actually be what uses more space.

\item
Our first implementation used PARI/GP, \cite{PARI2}, it took roughly two months of cpu time and $38Gb$ of ram on an Intel(R) Xeon(R) CPU E5-2690 to compute the values given up to $3\cdot 10^6$. 

We re-implemented the main parts of the algorithm in c, using the PARI c library. 
The improved version computed $P(n)$ for the range $3\cdot 10^6$  to $5\cdot 10^6$ million using $10.7$Gb in $137.9$h of cpu time. 
We computed values from $5\cdot 10^6$ million to $10^7$ million using $23.5$Gb and $630.9$h on the same system.
We computed values from $10^7$ million to $1.5\cdot 10^7 $ million using $47.3$Gb and $1125.7$h on the same system.

For the final computations we needed to further optimize memory usage and improve parallelism.
We used $114.4Gb$ and approximately $38000$h of cpu time on the same system to compute the values between $5\cdot 10^7$ and $10^8$.\footnote{Frequent power outages caused us to lose metadata, but not the results, for the computations between $1.5\cdot 10^7$ and $5\cdot 10^7$}

There remains room to improve both the memory usage and the use of parallelism in our implementation and this would be desirable before pushing the calculations significantly further.
\end{enumerate}
\end{rmk}

         \newpage
         
\begin{landscape}
\begin{table}[h]
\centering
\caption{Counts of large outliers with comparisons to various plausible expectations. This table is for $N=10^8$.
Note that for $\epsilon=1$ the third, forth and fifth column's are respectively $\log(N)$, $1$, and $\log(N)\log\log(N)$.
The value in parenthesis is the ratio comparing to the true count.
}\label{table:stats-large}
\begin{tabular}{|r|r|rrrrrr|}
\hline
$\epsilon$ & 
$|E_\epsilon(3,N)|$  &  
$(1-\epsilon)(N^{1-\epsilon}-1)$ &
$(1-\epsilon)\frac{N^{1-\epsilon}-1}{\log(N)}$ &
$(1-\epsilon)\log\log(N)(N^{1-\epsilon} -1)$ &
$ \sum_{k=1}^N \frac{1}{\phi(k)^{\epsilon}}$ &
$ \sum_{k=1}^N \frac{1}{k^{\epsilon}}$ &
$ \sum_{k=1}^N \frac{\phi(k)}{k^{\epsilon+1}}$\\
\hline
$0.0$ & $63512264$ & $99999999.00$ ($1.57$) & $5428680.97$  ($0.09$) &$291347395.78$  ($4.59$) & $99999998.00$  ($1.57$) & $99999998.00$ ($1.57$)  & $60792708.76$ ($0.96$)\\
$0.1$ & $19112007$ & $17609923.25$ ($0.92$) & $955986.56$  ($0.05$) &$51306053.30$  ($2.68$) & $18677431.85$  ($0.98$) & $17609921.90$ ($0.92$)  & $10705548.77$ ($0.56$)\\
$0.2$ & $4124995$ & $3139856.79$ ($0.76$) & $170452.81$  ($0.04$) &$9147891.08$  ($2.22$) & $3538070.83$  ($0.86$) & $3139855.45$ ($0.76$)  & $1908803.23$ ($0.46$)\\
$0.3$ & $832336$ & $568723.10$ ($0.68$) & $30874.16$  ($0.04$) &$1656959.96$  ($1.99$) & $682012.03$  ($0.82$) & $568721.81$ ($0.68$)  & $345741.42$ ($0.42$)\\
$0.4$ & $167725$ & $105157.89$ ($0.63$) & $5708.69$  ($0.03$) &$306374.78$  ($1.83$) & $134431.65$  ($0.80$) & $105156.67$ ($0.63$)  & $63927.60$ ($0.38$)\\
$0.5$ & $34331$ & $19998.00$ ($0.58$) & $1085.63$  ($0.03$) &$58263.65$  ($1.70$) & $27298.10$  ($0.80$) & $19996.83$ ($0.58$)  & $12156.63$ ($0.35$)\\
$0.6$ & $7111$ & $3959.73$ ($0.56$) & $214.96$  ($0.03$) &$11536.58$  ($1.62$) & $5780.22$  ($0.81$) & $3958.62$ ($0.56$)  & $2406.56$ ($0.34$)\\
$0.7$ & $1551$ & $833.96$ ($0.54$) & $45.27$  ($0.03$) &$2429.73$  ($1.57$) & $1302.88$  ($0.84$) & $832.90$ ($0.54$)  & $506.35$ ($0.33$)\\
$0.8$ & $352$ & $194.05$ ($0.55$) & $10.53$  ($0.03$) &$565.37$  ($1.61$) & $323.90$  ($0.92$) & $193.04$ ($0.55$)  & $117.37$ ($0.33$)\\
$0.9$ & $88$ & $53.10$ ($0.60$) & $2.88$  ($0.03$) &$154.69$  ($1.76$) & $93.84$  ($1.07$) & $52.13$ ($0.59$)  & $31.70$ ($0.36$)\\
$1.0$ & $25$ & $18.42$ ($0.74$) & $1.00$  ($0.04$) &$53.67$  ($2.15$) & $33.74$  ($1.35$) & $17.50$ ($0.70$)  & $10.65$ ($0.43$)\\
$1.1$ & $14$ & $8.42$ ($0.60$) & $0.46$  ($0.03$) &$24.52$  ($1.75$) & $15.51$  ($1.11$) & $7.53$ ($0.54$)  & $4.59$ ($0.33$)\\
$1.2$ & $9$ & $4.87$ ($0.54$) & $0.26$  ($0.03$) &$14.20$  ($1.58$) & $8.84$  ($0.98$) & $4.03$ ($0.45$)  & $2.46$ ($0.27$)\\
$1.3$ & $4$ & $3.32$ ($0.83$) & $0.18$  ($0.05$) &$9.67$  ($2.42$) & $5.85$  ($1.46$) & $2.51$ ($0.63$)  & $1.53$ ($0.38$)\\
$1.4$ & $3$ & $2.50$ ($0.83$) & $0.14$  ($0.05$) &$7.28$  ($2.43$) & $4.27$  ($1.42$) & $1.73$ ($0.58$)  & $1.05$ ($0.35$)\\
$1.5$ & $3$ & $2.00$ ($0.67$) & $0.11$  ($0.04$) &$5.83$  ($1.94$) & $3.31$  ($1.10$) & $1.26$ ($0.42$)  & $0.77$ ($0.26$)\\
\hline
\end{tabular}
\end{table}



\begin{table}[h]
\caption{Counts of small outliers with comparisons to our conjectured assymptotic bound. This table is for $N=10^8$.}\label{table:stats-small}
{
\begin{tabular}{|r|r|rrr|}
\hline
$\epsilon$ & $|F_\epsilon(3,N)|$  &  $\int e^{-k^\epsilon/(\log k\log\log k)}{\rm dk}$ &  $\int e^{-k^\epsilon/(\log k)}{\rm dk}$  &  $\int e^{-k^\epsilon/(\log\log k)}{\rm dk}$ \\
\hline
$0.0$ & $36487734$ &$97998679.90$ ($2.69$) & $94401358.22$ ($2.59$) &$70447724.18$ ($1.93$) \\
$0.1$ & $2164082$ &$89136078.96$ ($41.19$) & $72001735.39$ ($33.27$) &$13616607.00$ ($6.29$) \\
$0.2$ & $12027$ &$51781016.45$ ($4305.40$) & $15555547.31$ ($1293.39$) &$13978.58$ ($1.16$) \\
$0.3$ & $663$ &$2970248.50$ ($4480.01$) & $53206.52$ ($80.25$) &$74.73$ ($0.11$) \\
$0.4$ & $102$ &$12962.17$ ($127.08$) & $638.85$ ($6.26$) &$4.95$ ($0.05$) \\
$0.5$ & $17$ &$448.27$ ($26.37$) & $52.74$ ($3.10$) &$0.61$ ($0.04$) \\
$0.6$ & $2$ &$46.10$ ($23.05$) & $9.27$ ($4.64$) &$0.08$ ($0.04$) \\
$0.7$ & $0$ &$7.49$ ($-$) & $2.13$ ($-$) &$0.01$ ($-$) \\
\hline
\end{tabular}
}
\end{table}
\end{landscape}

{\small
\begin{longtable}{|r|rrrrr|}
\caption{List of large outliers for $3 \leq n< 10^8$}\label{table:largeoutliers}\\
\hline
$n$ & $\phi(n)$ & $a$ & $P(n)$ &$P(n)/\phi(n)\log(n)\log(\phi(n))$ & factorization \\
\hline
\endhead
\hline\endfoot
\input{./largetable.tex}
\end{longtable}
}

{\small
\begin{longtable}{|r|rrrrr|}
\caption{List of small outliers for $3 \leq n< 10^8$}\label{table:smalloutliers}\\
\hline
$n$ & $\phi(n)$ & $a$ & $P(n)$ &$P(n)/\phi(n)\log(n)\log(\phi(n))$ & factorization \\
\hline
\endhead
\hline\endfoot
$102$&$32$&$95$&$197$&$0.3841$&$2\cdot3\cdot17$\\
$150$&$40$&$143$&$293$&$0.3963$&$2\cdot3\cdot5^{2}$\\
$210$&$48$&$209$&$419$&$0.4217$&$2\cdot3\cdot5\cdot7$\\
$130$&$48$&$129$&$389$&$0.4301$&$2\cdot5\cdot13$\\
$51$&$32$&$44$&$197$&$0.4518$&$3\cdot17$\\
$420$&$96$&$361$&$1201$&$0.4538$&$2^{2}\cdot3\cdot5\cdot7$\\
$228$&$72$&$77$&$761$&$0.4552$&$2^{2}\cdot3\cdot19$\\
$246$&$80$&$145$&$883$&$0.4575$&$2\cdot3\cdot41$\\
\end{longtable}
}


\begin{thebibliography}{99}


\bibitem{Erdos}  P. Erdos, 
{\it On some applications of Brun's method},
Acta Sci. Math. (Szeged) {\bf v. 13.1}.
(1949), 
pp. 57--63.
\\\url{http://acta.bibl.u-szeged.hu/13658/1/math_013_fasc_001_057-063.pdf}


\bibitem{HB} D.R. Heath-Brown,, 
{\it Almost-primes in arithmetic progressions and short intervals},
Mathematical Proceedings of the Cambridge Philosophical Society {\bf v. 83.3}..
(1978),
pp. 357--375
\\\url{https://doi.org/10.1017/S0305004100054657}

\bibitem{Kanold} H.-J. Kanold,
{\it Über Primzahlen in arithmetischen Folgen},
Mathematische Annalen {\bf 156}.
(1964),
pp. 393--395.
\\\url{https://doi.org/10.1007/BF01362452}

\bibitem{Landau} E. Landau,
{\it Handbuch der Lehre von der Verteilung der Primzahlen},
Band I, Teubner, Leipzig-Berlin, 1909. Reprinted by Chelsea, New Yord, 1953.


\bibitem{Linnik} U. V. Linnik,
{\it On the least prime in an arithmetic progression. I. The basic theorem},
Rec. Math. [Mat. Sbornik] N.S., {\bf 15 (57)}. 
(1944), 
pp. 139--178.
\\\url{http://resolver.sub.uni-goettingen.de/purl?PPN477674380_0057}

\bibitem{PARI2}
    The PARI~Group, PARI/GP version \texttt{2.13.4}, Univ. Bordeaux, 2022,
    \url{http://pari.math.u-bordeaux.fr/}.

\bibitem{Pomerance} C. Pomerance, 
{\it A note on the Least Prime in an Arithmetic Progression},
Journal of Number Theory, {\bf 12}.
(1980),
pp. 218-223.
\\\url{https://doi.org/10.1016/0022-314X(80)90056-6}

\bibitem{Prachar} K. Prachar,
{\it Über die kleinste Primzahl einer arithmetischen Reihe},
Reine Angew. Math. {\b 206}.
(1961),
pp. 3--4.
\\\url{https://doi.org/10.1515/crll.1961.206.3}

\bibitem{Schinzel} A. Schinzel,
{\it Remark on the paper of K. Prachar 'Über die kleinste arithmetischen Reihe},
Reine Angew. Math. {\b 210}.
(1962),
pp. 121--122.
\\\url{https://doi.org/10.1515/crll.1962.210.121}

\bibitem{Titchmarsh} E. C. Titchmarsh, 
{\it A Divisor Problem}, 
Rend. Cire. Mat. Palermo, {\b 54}.
(1930),
pp. 414--429.
\\\url{https://doi.org/10.1007/BF03021203}

\bibitem{Turan} P. Tur\:an,
{\it Über die Primzahlen der arithmetischen Progression},
Acta Sci. Math. (Szeged) {\b 8}.
(1936/37), 
pp. 226-235.
\\\url{http://acta.bibl.u-szeged.hu/13483/1/math_008_226-235.pdf}

\bibitem{WAG79} S. S. Wagstaff, Jr, 
{\it Greatest of the Least Primes in Arithmetic Progressions Having a Given Modulus},
 Mathematics of Computation {\bf 33}. No. 147, 
 (1979), pp, 1073--1080.
 \\\url{https://doi.org/10.2307/2006082}
 


\bibitem{Xylouris} T. Xylouris, 
{\it Über die Nullstellen der Dirichletschen L-Funktionen und die kleinste Primzahl in einer arithmetischen Progression}, 
Bonn, 2011. - Dissertation, Rheinische Friedrich-Wilhelms-Universität Bonn. 
\\\url{https://nbn-resolving.org/urn:nbn:de:hbz:5N-27156 }



\end{thebibliography}
\end{document}